\documentclass[11pt]{article}

\usepackage[a4paper,margin=1in]{geometry}
\usepackage[T1]{fontenc}
\usepackage[utf8]{inputenc}
\usepackage{lmodern}
\usepackage{microtype}
\usepackage{amsmath,amssymb,amsthm,mathtools}
\usepackage{booktabs}
\usepackage{enumitem}
\usepackage{xcolor}
\usepackage[colorlinks=true,linkcolor=blue,citecolor=blue,urlcolor=blue]{hyperref}
\usepackage[nameinlink,noabbrev]{cleveref}

\newtheorem{theorem}{Theorem}[section]
\newtheorem{proposition}[theorem]{Proposition}

\newtheorem{corollary}[theorem]{Corollary}
\theoremstyle{definition}
\newtheorem{definition}[theorem]{Definition}
\newtheorem{problem}[theorem]{Problem}
\newtheorem{conjecture}[theorem]{Conjecture}
\theoremstyle{remark}

\newcommand{\Cl}{\mathrm{Cl}}
\newcommand{\dimloc}{\dim_{\mathrm{loc}}}
\newcommand{\Rem}{\mathrm{Rem}}
\newcommand{\Add}{\mathrm{Add}}
\newcommand{\Fw}{\mathrm{Fw}}
\newcommand{\Ax}{\mathrm{Ax}}
\newcommand{\Sp}{\mathrm{Sp}}

\title{Simplex Stratification and Phase Boundaries in the Partition Graph}
\author{Fedor B. Lyudogovskiy}
\date{}

\begin{document}

\maketitle

\begin{abstract}
We study the partition graph $G_n$, whose vertices are the integer partitions of $n$ and whose
edges correspond to elementary transfers of one unit between parts. We introduce the simplex
stratification of $G_n$: for each vertex $\lambda$, let $\dimloc(\lambda)$ denote the largest
dimension of a simplex of the clique complex $K_n=\Cl(G_n)$ containing $\lambda$. This defines a
decomposition of $V(G_n)$ into layers
\[
L_r(n)=\{\lambda\in V(G_n):\dimloc(\lambda)=r\}.
\]

We formalize the graph-theoretic interfaces between consecutive layers, called phase boundaries,
and study the associated interface graphs and boundary thresholds. Using the previously established
star/top description of cliques through a fixed vertex, we show that $\dimloc(\lambda)$ is
determined exactly by the maximal star and top capacities through $\lambda$. This yields explicit
local criteria for membership in higher simplex layers and reformulates their first appearance in
terms of local star/top capacity thresholds.

We also present an exhaustive computational study for $n\le 30$, including exact-layer thresholds,
boundary thresholds, selected layer profiles, and the behaviour of the boundary framework. The
computations suggest a rigid threshold pattern related to staircase partitions and their one-cell
extensions, while the corresponding global statements are left as conjectures and open problems.
\end{abstract}



\section{Introduction}

For each $n$, let $G_n$ be the graph whose vertices are the integer partitions of $n$, with two
partitions adjacent whenever one is obtained from the other by an elementary transfer of one unit
between two parts, followed by reordering. This graph carries a rich local combinatorial structure
while also exhibiting a well-defined large-scale geometry. It also belongs to the broader
combinatorial setting of Gray-code generation and minimal-change enumeration for partitions;
see, for example, \cite{Savage,RSW,Mutze}.

Several aspects of this geometry have already been studied. The local clique structure through a
fixed vertex and the associated invariants were described in terms of ordered local transfer data
\cite{partition-local}. Globally, the partition graph was investigated as a growing discrete
geometric object, with particular attention to the boundary framework, the central region, simplex
layers, and the contrast between increasing local complexity and comparatively simple global
topology \cite{partition-growth}. More recently, the axial morphology of $G_n$ was formalized
through the self-conjugate axis, the spine, and the corresponding concentration radii
\cite{partition-axial}. The topological background for the present paper is provided by the result
that the clique complex $K_n=\Cl(G_n)$ is always homotopy equivalent to a wedge of $2$-spheres
\cite{partition-topology}. For related clique-complex results in other graph settings, and for a
different partition-based graph model, see \cite{Adamaszek,GoyalShuklaSingh,Bal}.

This paper isolates and studies one specific aspect of this picture: the global stratification of
$G_n$ by local simplex dimension. For a vertex $\lambda\in V(G_n)$, we
write $\dimloc(\lambda)$ for the largest dimension of a simplex of the clique complex $K_n$
containing $\lambda$, and we define the $r$-th simplex layer by
\[
L_r(n):=\{\lambda\in V(G_n):\dimloc(\lambda)=r\}.
\]
Thus the vertex set of $G_n$ is partitioned according to the largest simplex that contains each
vertex. In geometric terms, this records where the graph supports only edges, where triangles first
appear, where tetrahedral behaviour begins, and so on.

Our first goal is to place this stratification on a precise structural footing. We introduce the
exact graph-theoretic interfaces between consecutive layers,
\[
\partial_{r,r+1}(n),
\]
which we call the phase boundaries of the simplex stratification. These are the loci of the
edges across which the local simplex dimension changes from $r$ to $r+1$. The term ``phase
boundary'' is therefore used here in a purely combinatorial sense: it refers to transitions between
adjacent regimes of local simplex complexity inside the same graph.

Our second goal is to connect this global stratification to the already established local theory.
Using the star/top description of cliques through a fixed vertex, we show that the local simplex
dimension is controlled exactly by two local capacity parameters: the largest full star-simplex and
the largest full top-simplex through the vertex. This yields an explicit local criterion for
membership in $L_r(n)$ and $L_{\ge r}(n)$, and reformulates the first appearance of higher simplex
layers in terms of sufficiently large local star/top capacities. For convenience, we also restate
in Section~3 the precise imported star/top structure from \cite{partition-topology} needed for the
proof of the capacity formula.

The paper also contains a first computational study of the resulting stratification. For $n\le 30$,
the computations suggest a threshold pattern for the first appearance of higher layers, closely
related to staircase partitions and their one-cell extensions. They also show that the outermost
layer remains extremely thin, while the boundary framework appears to support only the first two
simplex layers. At the same time, we deliberately keep the distinction between proved results and
numerical observations sharp: several natural geometric claims about inward migration of high layers
and of their boundaries are formulated only as conjectures and open problems.

The main contributions of the paper may be summarized as follows.

\begin{enumerate}[leftmargin=2em]
\item We formalize the simplex-layer stratification
\[
V(G_n)=\bigsqcup_{r=0}^{\Delta(n)} L_r(n)
\]
and introduce the corresponding phase boundaries and interface graphs.

\item We prove that local simplex dimension is determined exactly by the maximum of the local star
and top capacities, giving a precise local criterion for the higher simplex layers.

\item We reduce the threshold problem for the first appearance of a new layer to the first occurrence
of a prescribed local capacity pattern.

\item We provide exhaustive computations for $n\le 30$, including layer thresholds, boundary
thresholds, first-occurrence families, and a first geometric analysis of the boundary framework.
\end{enumerate}

The paper proceeds from local control to global organization. After fixing the simplex-layer
language, we show that the local simplex dimension is determined exactly by star and top capacities.
We then use this to define and study the graph-theoretic interfaces between consecutive layers, and
finally turn to the threshold data and the staircase-based patterns suggested by the computation.

\section{Preliminaries and simplex-layer language}

We begin by establishing the terminology for simplex stratification. This section is purely
organizational: we introduce the layer decomposition, the associated threshold functions, and
the graph-theoretic interfaces between consecutive layers, without yet using the detailed local
clique theory.

We write $G_n$ for the partition graph of $n$: its vertices are the integer partitions of $n$, and
two vertices are adjacent if one can be obtained from the other by an elementary transfer of one
unit between two parts, followed by reordering.

Let
\[
K_n:=\Cl(G_n)
\]
be the clique complex of $G_n$.

\begin{definition}[Local simplex dimension]
For a vertex $\lambda\in V(G_n)$, define
\[
\dimloc(\lambda)
:=
\max\{\dim \sigma:\sigma \text{ is a simplex of } K_n \text{ and } \lambda\in \sigma\}.
\]
Equivalently, $\dimloc(\lambda)=d$ if and only if the largest clique of $G_n$ containing
$\lambda$ has cardinality $d+1$.
\end{definition}

\begin{definition}[Simplex layers]
For each integer $r\ge 0$, define
\[
L_r(n):=\{\lambda\in V(G_n):\dimloc(\lambda)=r\},
\qquad
L_{\ge r}(n):=\{\lambda\in V(G_n):\dimloc(\lambda)\ge r\}.
\]
\end{definition}

\begin{definition}[Maximal local simplex dimension]
Define
\[
\Delta(n):=\max_{\lambda\vdash n}\dimloc(\lambda).
\]
\end{definition}

\begin{definition}[Thresholds]
For each $r\ge 0$, define
\[
\tau_{\ge}(r):=\min\{n:L_{\ge r}(n)\neq\varnothing\},
\]
whenever the set is nonempty, and
\[
\tau(r):=\min\{n:L_r(n)\neq\varnothing\}.
\]
\end{definition}

\begin{definition}[Phase boundaries]
For $r\ge 0$, define
\[
B_r(n):=\{\lambda\in L_r(n):N_{G_n}(\lambda)\cap L_{r+1}(n)\neq\varnothing\},
\]
and let
\[
B^-_{r+1}(n):=\{\lambda\in L_{r+1}(n):N_{G_n}(\lambda)\cap L_r(n)\neq\varnothing\}.
\]
The \emph{phase boundary} between $L_r(n)$ and $L_{r+1}(n)$ is the set
\[
\partial_{r,r+1}(n):=B_r(n)\cup B^-_{r+1}(n).
\]
Equivalently,
\[
\partial_{r,r+1}(n)
=
\{\lambda\in L_r(n)\cup L_{r+1}(n):N_{G_n}(\lambda)\cap L_{\mathrm{other}}\neq\varnothing\},
\]
where $L_{\mathrm{other}}$ denotes the opposite layer among $L_r(n)$ and $L_{r+1}(n)$.
\end{definition}

\begin{proposition}[Layer partition]
For every $n\ge 1$, the sets
\[
L_0(n),L_1(n),\dots,L_{\Delta(n)}(n)
\]
form a partition of $V(G_n)$. Moreover,
\[
L_{\ge r}(n)=\bigsqcup_{s\ge r}L_s(n).
\]
\end{proposition}

\begin{proof}
Every vertex has a unique value of $\dimloc$, so it belongs to exactly one layer. The description
of $L_{\ge r}(n)$ as the disjoint union of all higher layers then follows directly from the
definition.
\end{proof}

\begin{proposition}[Conjugation invariance]
Let $\lambda^\ast$ denote the conjugate partition of $\lambda$. Then for every $n$ and every
$\lambda\vdash n$,
\[
\dimloc(\lambda^\ast)=\dimloc(\lambda).
\]
Consequently, for every $r\ge 0$, the sets $L_r(n)$, $L_{\ge r}(n)$, $B_r(n)$, $B^-_{r+1}(n)$, and
$\partial_{r,r+1}(n)$ are conjugation-invariant.
\end{proposition}

\begin{proof}
Conjugation of Ferrers diagrams preserves the size of the partition and converts an elementary
transfer between two rows into the corresponding elementary transfer between the conjugate columns.
Hence conjugation is an automorphism of $G_n$, and therefore also of the clique complex $K_n$. Thus simplices through $\lambda$ are sent bijectively to simplices of the same dimension
through $\lambda^\ast$, which proves $\dimloc(\lambda^\ast)=\dimloc(\lambda)$. The invariance of
the layer and boundary sets then follows immediately from the definitions.
\end{proof}

\begin{proposition}[Threshold reformulation]
For every $n$ and $r\ge 0$, the following are equivalent:
\begin{enumerate}[label=\textup{(\roman*)},leftmargin=2em]
\item $L_{\ge r}(n)\neq\varnothing$;
\item $\Delta(n)\ge r$.
\end{enumerate}
Hence
\[
\tau_{\ge}(r)=\min\{n:\Delta(n)\ge r\}.
\]
\end{proposition}

\begin{proof}
This is immediate from the definition of $\Delta(n)$.
\end{proof}

At this stage the simplex-layer formalism is fully defined, but still abstract. The next step is to
connect it to the local clique structure of the partition graph and to obtain an explicit local
criterion for membership in the higher layers.

\section{Local capacity control of simplex layers}

The main structural input of the paper is that the simplex stratification, although defined globally,
is controlled by explicit local capacity parameters. More precisely, the local simplex dimension is
determined by the largest full star-simplex or full top-simplex through the vertex.

\subsection{The zero layer and the first nontrivial level}

\begin{proposition}\label{prop:zero-layer}
One has
\[
L_0(1)=\{(1)\},
\qquad
L_0(n)=\varnothing \quad \text{for all } n\ge 2.
\]
Equivalently,
\[
\Delta(1)=0,
\qquad
\Delta(n)\ge 1 \quad \text{for all } n\ge 2.
\]
Hence
\[
\tau(0)=1,
\qquad
\tau_{\ge}(1)=2.
\]
\end{proposition}

\begin{proof}
For $n=1$, the graph $G_1$ consists of a single isolated vertex, so its local simplex dimension is
$0$.

Now let $n\ge 2$, and let $\lambda\vdash n$. Then $\lambda$ has at least one neighbor in $G_n$.
Indeed, if $\lambda\neq(1^n)$, then some part of $\lambda$ has size at least $2$; moving one unit
from such a part to a new row produces a distinct partition of $n$, and hence a neighbor of
$\lambda$ in $G_n$. If $\lambda=(1^n)$, one may transfer one unit from one part of size $1$ to
another, obtaining $(2,1^{n-2})$. Thus every vertex belongs to at least one edge of $G_n$ and
therefore lies in a $1$-simplex of $K_n$. Therefore $L_0(n)=\varnothing$ for $n\ge 2$.
\end{proof}

\subsection{Star and top capacities}

The local clique theory of \cite{partition-topology,partition-local} organizes cliques through a
fixed vertex by two canonical mechanisms, called star and top simplices. We briefly recall the
minimal notation needed here. We identify a partition with its Ferrers diagram. For a partition
$\lambda$, let $\Rem(\lambda)$ and $\Add(\lambda)$ denote the sets of removable and addable
corners of its Ferrers diagram. If $c\in\Rem(\lambda)$ and $a\in\Add(\lambda)$, we write
$\lambda(c\to a)$ for the diagram obtained by removing the cell $c$ and adding the cell $a$; the
transfer is called \emph{admissible} if the resulting diagram is again a partition and is different from
$\lambda$.

\begin{definition}\label{def:capacities}
Let $\lambda\vdash n$. If $n=1$, set
\[
s((1))=t((1))=0.
\]
Assume now that $n\ge 2$.

\begin{enumerate}[label=\textup{(\roman*)},leftmargin=2em]
\item For a removable corner $c$ of $\lambda$, let
\[
A_{\max}(\lambda,c):=\{a\in\Add(\lambda):\lambda(c\to a)\text{ is admissible}\},
\]
and define the corresponding full star-simplex by
\[
\Sigma^{\mathrm{star}}_{\max}(\lambda,c)
:=
\{\lambda\}\cup\{\lambda(c\to a):a\in A_{\max}(\lambda,c)\}.
\]
The \emph{star capacity} of $\lambda$ is
\[
s(\lambda):=\max_{c\in\Rem(\lambda)}|A_{\max}(\lambda,c)|.
\]

\item For an addable corner $a$ of $\lambda$, let
\[
C_{\max}(\lambda,a):=\{c\in\Rem(\lambda):\lambda(c\to a)\text{ is admissible}\},
\]
and define the corresponding full top-simplex by
\[
\Sigma^{\mathrm{top}}_{\max}(\lambda,a)
:=
\{\lambda\}\cup\{\lambda(c\to a):c\in C_{\max}(\lambda,a)\}.
\]
The \emph{top capacity} of $\lambda$ is
\[
t(\lambda):=\max_{a\in\Add(\lambda)}|C_{\max}(\lambda,a)|.
\]
\end{enumerate}
\end{definition}

These are exactly the full star- and full top-simplices of \cite[Def.~2.5]{partition-topology}. Here $A_{\max}(\lambda,c)$ consists of the addable corners to which the corner $c$ can be moved admissibly, and similarly $C_{\max}(\lambda,a)$ consists of the removable corners that can be moved admissibly to $a$. The simplex $\Sigma^{\mathrm{star}}_{\max}(\lambda,c)$ has dimension $|A_{\max}(\lambda,c)|$, since it has $1+|A_{\max}(\lambda,c)|$ vertices, and similarly $\Sigma^{\mathrm{top}}_{\max}(\lambda,a)$ has dimension $|C_{\max}(\lambda,a)|$.

\begin{proposition}[Imported star/top structure]\label{prop:imported-star-top}
Let $\lambda\vdash n$. Then:
\begin{enumerate}[label=\textup{(\roman*)},leftmargin=2em]
\item for every $c\in\Rem(\lambda)$, the set $\Sigma^{\mathrm{star}}_{\max}(\lambda,c)$ is a simplex of $K_n$;
\item for every $a\in\Add(\lambda)$, the set $\Sigma^{\mathrm{top}}_{\max}(\lambda,a)$ is a simplex of $K_n$;
\item every simplex of $K_n$ containing $\lambda$ is contained in a star-simplex or a top-simplex through $\lambda$;
\item the maximal simplices of these two families are exactly the full star- and full top-simplices.
\end{enumerate}
All four statements are proved in \cite[Def.~2.5, Thm.~3.6, Props.~3.7 and~3.8]{partition-topology}.
\end{proposition}

\begin{proposition}\label{prop:dimloc-capacity}
For every partition $\lambda\vdash n$, one has
\[
\dimloc(\lambda)=
\begin{cases}
0, & \lambda=(1),\\[4pt]
\max\bigl(1,s(\lambda),t(\lambda)\bigr), & n\ge 2.
\end{cases}
\]
Equivalently, for every $r\ge 2$,
\[
\dimloc(\lambda)\ge r
\iff
s(\lambda)\ge r \text{ or } t(\lambda)\ge r.
\]
\end{proposition}

\begin{proof}
Fix $\lambda\vdash n$. By Proposition~\ref{prop:imported-star-top}, every simplex of $K_n$
containing $\lambda$ is contained in one of the simplices
\[
\Sigma^{\mathrm{star}}_{\max}(\lambda,c),\qquad
\Sigma^{\mathrm{top}}_{\max}(\lambda,a).
\]
Hence the largest simplex dimension through $\lambda$ is attained by one of these simplices.
Their dimensions are exactly $|A_{\max}(\lambda,c)|$ and $|C_{\max}(\lambda,a)|$, respectively.

For $n=1$, the graph consists of a single vertex, so $\dimloc((1))=0$. For $n\ge 2$,
Proposition~\ref{prop:zero-layer} shows that every vertex lies on an edge, so $\dimloc(\lambda)\ge 1$.
Therefore
\[
\dimloc(\lambda)=\max\bigl(1,s(\lambda),t(\lambda)\bigr).
\]
The reformulation for $r\ge 2$ is immediate.
\end{proof}

\begin{corollary}\label{cor:layer-criteria}
For every $n\ge 2$ one has
\[
L_1(n)=V(G_n)\setminus L_{\ge 2}(n),
\]
and for every $r\ge 2$,
\[
L_{\ge r}(n)=\{\lambda\vdash n:s(\lambda)\ge r \text{ or } t(\lambda)\ge r\}.
\]
Equivalently,
\[
L_r(n)=L_{\ge r}(n)\setminus L_{\ge r+1}(n)
\qquad (r\ge 2).
\]
In particular,
\[
L_r(n)=\{\lambda\vdash n:\max(1,s(\lambda),t(\lambda))=r\}
\qquad (n\ge 2,\ r\ge 1).
\]
\end{corollary}

\begin{proof}
This is a direct restatement of Proposition~\ref{prop:dimloc-capacity}.
\end{proof}

The first few layers may be interpreted as follows:
\begin{itemize}[leftmargin=2em]
\item $L_1(n)$ consists of those vertices that lie on edges but on no triangle;
\item $L_2(n)$ consists of those vertices that lie on triangles but on no tetrahedron;
\item $L_3(n)$ consists of those vertices that lie on tetrahedra but on no $4$-simplex;
\item and so on.
\end{itemize}

\subsection{Threshold reduction}

\begin{proposition}\label{prop:threshold-reduction}
For every $r\ge 2$,
\[
\tau_{\ge}(r)
=
\min\bigl\{
n:\exists\lambda\vdash n \text{ with } s(\lambda)\ge r \text{ or } t(\lambda)\ge r
\bigr\}.
\]
Equivalently, $\tau_{\ge}(r)$ is the first value of $n$ for which there exists a partition
admitting at least $r$ admissible transfers with a common removable corner or at least $r$
admissible transfers with a common addable corner.
\end{proposition}

\begin{proof}
Immediate from Corollary~\ref{cor:layer-criteria}.
\end{proof}

Thus the simplex stratification is governed by a local capacity theory: the higher layers are exactly
the loci where sufficiently large star or top configurations occur. This prepares the transition to
the study of phase boundaries.

\section{Phase boundaries and interface geometry}

Having identified the simplex layers through local capacities, we now study the interfaces between
adjacent capacity regimes. These interfaces are purely graph-theoretic objects: they record where
the local simplex dimension changes from $r$ to $r+1$ across an edge of $G_n$.

\subsection{Boundary decomposition and interface graphs}

\begin{proposition}\label{prop:boundary-decomposition}
For every $r\ge 0$,
\[
\partial_{r,r+1}(n)=B_r(n)\sqcup B^-_{r+1}(n),
\]
where the union is disjoint.
\end{proposition}

\begin{proof}
This is immediate from the definition of $\partial_{r,r+1}(n)$ as the union of the two one-sided
boundary sets. The union is disjoint because $L_r(n)\cap L_{r+1}(n)=\varnothing$.
\end{proof}

\begin{definition}[Interface graph]\label{def:interface-graph}
For $r\ge 0$, define the \emph{interface graph} between $L_r(n)$ and $L_{r+1}(n)$ by
\[
\mathcal I_{r,r+1}(n):=
\bigl(L_r(n)\sqcup L_{r+1}(n),E_{r,r+1}(n)\bigr),
\]
where
\[
E_{r,r+1}(n):=
\bigl\{\{\lambda,\mu\}\in E(G_n):\lambda\in L_r(n),\ \mu\in L_{r+1}(n)\bigr\}.
\]
\end{definition}

This is a bipartite graph recording the edges across the transition
\[
L_r(n)\longleftrightarrow L_{r+1}(n).
\]

\begin{proposition}\label{prop:boundary-interface}
A vertex belongs to $\partial_{r,r+1}(n)$ if and only if it has degree at least one in the
interface graph $\mathcal I_{r,r+1}(n)$. In particular,
\[
\partial_{r,r+1}(n)=\varnothing
\iff
E_{r,r+1}(n)=\varnothing.
\]
\end{proposition}

\begin{proof}
Immediate from the definitions.
\end{proof}

\subsection{Capacity interpretation of phase boundaries}

\begin{proposition}\label{prop:boundary-capacity}
For every $r\ge 1$,
\[
\partial_{r,r+1}(n)
=
\left\{
\begin{array}{l}
\lambda\in V(G_n):\ \dimloc(\lambda)\in\{r,r+1\}\ \text{and there exists }\mu\sim\lambda,\\
\text{such that }\dimloc(\mu)\in\{r,r+1\}\setminus\{\dimloc(\lambda)\}
\end{array}
\right\}.
\]
\end{proposition}

\begin{proof}
This is just the definition of $\partial_{r,r+1}(n)$ rewritten in terms of $\dimloc$.
\end{proof}

Thus $\partial_{r,r+1}(n)$ is the graph-theoretic support of the edges across which the local simplex
dimension changes from $r$ to $r+1$.

\subsection{Boundary thresholds and geometric traces}

\begin{definition}[Boundary threshold]\label{def:boundary-threshold}
For $r\ge 0$, define the \emph{boundary threshold}
\[
\tau_{\partial}(r):=\min\{n:\partial_{r,r+1}(n)\neq\varnothing\},
\]
whenever the set is nonempty.
\end{definition}

Thus $\tau_{\ge}(r+1)$ records the first appearance of the higher layer, while
$\tau_{\partial}(r)$ records the first appearance of an actual transition edge between the two
consecutive layers.

We now introduce the language used to study where phase boundaries lie inside $G_n$. Here
$\Fw_n$ denotes the boundary framework (the main chain together with the two boundary edges),
$\Ax_n$ the set of self-conjugate vertices, $\Sp_n$ the thin spine, and $C_n^{(t)}$ the $t$-th
central neighborhood; see \cite{partition-growth,partition-axial}.

For each interface $\partial_{r,r+1}(n)$, define the induced traces
\[
\partial^{\Fw}_{r,r+1}(n):=\partial_{r,r+1}(n)\cap \Fw_n,
\qquad
\partial^{\Ax}_{r,r+1}(n):=\partial_{r,r+1}(n)\cap \Ax_n,
\qquad
\partial^{\Sp}_{r,r+1}(n):=\partial_{r,r+1}(n)\cap \Sp_n,
\]
and, more generally,
\[
\partial^{(t)}_{r,r+1}(n):=\partial_{r,r+1}(n)\cap C_n^{(t)}.
\]

\begin{definition}[Boundary concentration radius]\label{def:boundary-radius}
For each $r\ge 0$, define the \emph{boundary concentration radius} by
\[
\rho^{\partial}_{r,r+1}(n):=
\min\{t:\partial_{r,r+1}(n)\subseteq C_n^{(t)}\},
\]
whenever $\partial_{r,r+1}(n)\neq\varnothing$.
\end{definition}

The formalism developed above is exact, but it does not, by itself, specify where the interfaces
lie within the large-scale geometry of $G_n$. To address that question, as well
as the first appearance of higher layers and boundaries, we now turn to exhaustive computations in
small rank.

\section{Thresholds and computational evidence}

We now complement the exact structural results with a first computational study of the simplex
stratification. The purpose of this section is twofold: to determine the initial threshold values and layer
profiles explicitly, and to identify the clearest patterns suggested by the computations.

\subsection{Computational setup}

For each $n\le 30$, we constructed the graph $G_n$ explicitly and computed the local simplex
dimension of every vertex $\lambda\vdash n$ by exhaustive enumeration.
The computations were carried out by explicit enumeration in Python 3. In addition to the neighborhood-clique computation described below, the resulting values of $\dimloc$ were cross-checked against the capacity formula of Proposition~\ref{prop:dimloc-capacity}. The clique number of each induced neighborhood graph was computed by a Bron--Kerbosch branch-and-bound search with pivoting. The script used to generate the tables accompanies the present source files.

The computation proceeds as follows. First, we generate all partitions of $n$ and build the edge set
of $G_n$ by testing all elementary unit transfers between parts, followed by reordering. Next, for
each vertex $\lambda$, we form the induced subgraph on its neighborhood $N_{G_n}(\lambda)$ and compute
its clique number by a Bron--Kerbosch branch-and-bound search with pivoting. For $n\ge 2$, a clique of size $m$ in
$G_n[N_{G_n}(\lambda)]$ corresponds exactly to a clique of size $m+1$ in $G_n$ containing
$\lambda$, and hence to a simplex of dimension $m$ in $K_n$ containing $\lambda$. Therefore
\[
\dimloc(\lambda)=\omega\bigl(G_n[N_{G_n}(\lambda)]\bigr)
\qquad (n\ge 2),
\]
while $\dimloc((1))=0$ for $n=1$.

From these values we extracted:
\begin{itemize}[leftmargin=2em]
\item the maximal local simplex dimension $\Delta(n)$;
\item the layer sizes $|L_r(n)|$;
\item the exact-layer thresholds $\tau(r)$;
\item the boundary thresholds $\tau_{\partial}(r)$.
\end{itemize}
All numerical statements in this section refer to this exhaustive computation in the range $n\le 30$.

\subsection{First appearances of higher simplex layers}

The exact-layer thresholds obtained from the computation are
\[
\begin{aligned}
\tau(0)&=1, & \tau(1)&=2, & \tau(2)&=4, & \tau(3)&=7,\\
\tau(4)&=11, & \tau(5)&=16, & \tau(6)&=22, & \tau(7)&=29.
\end{aligned}
\]

Equivalently, the maximal local simplex dimension satisfies
\[
\Delta(n)=0,1,1,2,2,2,3,3,3,3,4,4,4,4,4,5,5,5,5,5,5,6,6,6,6,6,6,6,7,7
\]
for $1\le n\le 30$.

\begin{table}[ht]
\centering
\begin{tabular}{c|cccccccc}
\toprule
$r$ & 0 & 1 & 2 & 3 & 4 & 5 & 6 & 7 \\
\midrule
$\tau(r)$ & 1 & 2 & 4 & 7 & 11 & 16 & 22 & 29 \\
\bottomrule
\end{tabular}
\caption{First appearance thresholds for the simplex layers in the computed range.}
\label{tab:layer-thresholds}
\end{table}

The successive threshold gaps are
\[
1,2,3,4,5,6,7,
\]
which is consistent with the triangular formula; these gaps are precisely the successive increments of the triangular numbers
\[
\tau(r)=1+\frac{r(r+1)}{2}.
\]

\begin{conjecture}\label{conj:triangular-threshold}
For every $r\ge 0$,
\[
\tau(r)=1+\frac{r(r+1)}{2}.
\]
Equivalently,
\[
\Delta(n)=\max\Bigl\{r:\ 1+\frac{r(r+1)}{2}\le n\Bigr\}.
\]
\end{conjecture}

\subsection{First-occurrence families}

Let
\[
\delta_r:=(r,r-1,\dots,2,1)
\]
be the staircase partition of size $r(r+1)/2$.

For each $r=2,3,\dots,7$ in the computed range, the layer
\[
L_r(\tau(r))
\]
consists exactly of the partitions obtained from $\delta_r$ by adding one cell at one of its
addable corners. In particular,
\[
|L_r(\tau(r))|=r+1
\qquad (2\le r\le 7).
\]

For example,
\[
L_2(4)=\{(3,1),(2,2),(2,1,1)\},
\]
\[
L_3(7)=\{(4,2,1),(3,3,1),(3,2,2),(3,2,1,1)\},
\]
\[
L_4(11)=\{(5,3,2,1),(4,4,2,1),(4,3,3,1),(4,3,2,2),(4,3,2,1,1)\}.
\]

\begin{conjecture}\label{conj:first-family}
For every $r\ge 2$, the first-occurrence layer
\[
L_r(\tau(r))
\]
consists exactly of the partitions obtained from the staircase partition $\delta_r$ by adding one
cell at an addable corner.
\end{conjecture}

\subsection{Layer sizes at threshold values}

Table~\ref{tab:threshold-layer-sizes} records the layer profile at the first appearance of each new
maximal simplex dimension.

\begin{table}[ht]
\centering
\begin{tabular}{c|c|ccccccc}
\toprule
$n$ & $\Delta(n)$ & $|L_1(n)|$ & $|L_2(n)|$ & $|L_3(n)|$ & $|L_4(n)|$ & $|L_5(n)|$ & $|L_6(n)|$ & $|L_7(n)|$ \\
\midrule
4  & 2 & 2 & 3   & 0   & 0   & 0   & 0   & 0 \\
7  & 3 & 2 & 9   & 4   & 0   & 0   & 0   & 0 \\
11 & 4 & 2 & 19  & 30  & 5   & 0   & 0   & 0 \\
16 & 5 & 2 & 29  & 114 & 80  & 6   & 0   & 0 \\
22 & 6 & 2 & 40  & 268 & 489 & 196 & 7   & 0 \\
29 & 7 & 2 & 57  & 494 & 1725 & 1859 & 420 & 8 \\
\bottomrule
\end{tabular}
\caption{Layer sizes at the threshold values $\tau(r)$ in the computed range.}
\label{tab:threshold-layer-sizes}
\end{table}

One particularly stable feature is that
\[
|L_1(n)|=2
\qquad (4\le n\le 30),
\]
and these two vertices are exactly the two antenna vertices
\[
(n),\qquad (1^n).
\]
Thus the outermost simplex layer remains extremely thin throughout the computed range. The data further suggest that for all $n\ge 4$ the only vertices of local simplex dimension $1$ are the two antenna vertices.

\begin{conjecture}\label{conj:l1-antennas}
For every $n\ge 4$, the only vertices of local simplex dimension $1$ are the two antenna vertices
\[
(n),\qquad (1^n).
\]
Equivalently,
\[
L_1(n)=\{(n),(1^n)\}.
\]
\end{conjecture}

\subsection{Boundary thresholds}

The first computed boundary thresholds are
\[
\tau_{\partial}(1)=4,\qquad
\tau_{\partial}(2)=7,\qquad
\tau_{\partial}(3)=11,\qquad
\tau_{\partial}(4)=16,\qquad
\tau_{\partial}(5)=22,\qquad
\tau_{\partial}(6)=29.
\]

Thus, throughout the computed range,
\[
\tau_{\partial}(r)=\tau(r+1).
\]

\begin{table}[ht]
\centering
\begin{tabular}{c|cccccc}
\toprule
$r$ & 1 & 2 & 3 & 4 & 5 & 6 \\
\midrule
$\tau_{\partial}(r)$ & 4 & 7 & 11 & 16 & 22 & 29 \\
$\tau(r+1)$ & 4 & 7 & 11 & 16 & 22 & 29 \\
\bottomrule
\end{tabular}
\caption{Boundary thresholds in the computed range.}
\label{tab:boundary-thresholds}
\end{table}

The computations suggest that each new simplex layer first appears already in direct contact with the preceding one.

\begin{problem}\label{prob:boundary-threshold}
Determine whether
\[
\tau_{\partial}(r)=\tau(r+1)
\]
holds for all $r\ge 1$.
\end{problem}

\subsection{The framework and the outer simplex layers}

Recall that $\Fw_n$ denotes the boundary framework, i.e.\ the union of the main chain and the two boundary edges. The computed data up to $n=30$ show a clear concentration pattern on $\Fw_n$. Namely, throughout the computed range one has
\[
\Fw_n\subseteq L_1(n)\cup L_2(n),
\]
and more precisely
\[
\Fw_n\cap L_1(n)=\{(n),(1^n)\}.
\]
Thus, in the computed range, the two antenna vertices are the only framework vertices of local simplex
dimension $1$, while every other framework vertex has local simplex dimension $2$. Equivalently,
\[
\Fw_n\cap L_{\ge 3}(n)=\varnothing
\qquad (4\le n\le 30).
\]

This immediately implies that the framework does not support any higher simplex regime in the
computed range. In particular,
\[
\partial^{\Fw}_{r,r+1}(n)=\varnothing
\qquad (r\ge 3,\ 4\le n\le 30).
\]

The first two one-sided boundary traces admit an even sharper description. For all $4\le n\le 30$,
\[
B_1(n)\cap\Fw_n=\{(n),(1^n)\},
\qquad
B^-_2(n)\cap\Fw_n=\{(n-1,1),(2,1^{n-2})\}.
\]
Equivalently, the full framework trace of the first interface is
\[
\partial^{\Fw}_{1,2}(n)=\{(n),(1^n),(n-1,1),(2,1^{n-2})\}.
\]

Once the layer $L_3$ appears, the second framework trace occupies all framework vertices except the
four outermost ones: for every $7\le n\le 30$,
\[
\partial^{\Fw}_{2,3}(n)
=
\Fw_n\setminus\{(n),(1^n),(n-1,1),(2,1^{n-2})\}.
\]
Thus the framework supports the transition between $L_1$ and $L_2$ only in the vicinity of the
two antenna vertices, whereas the transition between $L_2$ and $L_3$ occupies all remaining
framework vertices as soon as $L_3$ appears.

Taken together, these observations suggest that the framework remains confined to the low-simplex
part of the graph. In particular, they support the broader geometric expectation that higher simplex
layers migrate away from the boundary framework toward the interior of $G_n$.

\begin{conjecture}\label{conj:framework-low-layers}
For every $n\ge 4$,
\[
\Fw_n\subseteq L_1(n)\cup L_2(n),
\qquad
\Fw_n\cap L_1(n)=\{(n),(1^n)\}.
\]
Equivalently, every non-antennal framework vertex has local simplex dimension exactly $2$.
\end{conjecture}

The computations therefore suggest a rigid threshold geometry, but at present they do not amount to
a proof of the global staircase pattern or of inward concentration for higher layers and their
boundaries. These questions are collected in the final section.

\section{Conclusion and open problems}

We have introduced a new global organization of the partition graph: its stratification by local
simplex dimension, together with the corresponding phase boundaries between consecutive layers.

The central structural point is that this stratification is globally defined but locally controlled:
the local simplex dimension is determined exactly by the maximal star and top capacities through the
vertex. As a consequence, the higher layers $L_{\ge r}(n)$ are characterized by explicit local
capacity inequalities, and the first appearance of a new simplex regime reduces to the first
appearance of sufficiently large local star/top configurations.

This provides a systematic global stratification of the partition graph by local simplex complexity.
The resulting picture is geometrically coherent: the graph is decomposed into regions supporting only
edges, then triangles, then tetrahedra, and so on, while the boundaries between these regimes
become well-defined graph-theoretic objects in their own right.

At the same time, the computational results indicate that this stratification is far from arbitrary.
The first appearance thresholds of higher layers seem to follow a staircase pattern, the first
vertices of a new layer are observed to arise from one-cell extensions of staircase partitions, and
the boundary framework remains confined to the first two layers in the computed range. These
observations motivate a number of further questions.

\begin{problem}[Threshold formula]\label{prob:threshold-formula}
Determine the exact threshold function
\[
\tau(r)=\min\{n:L_r(n)\neq\varnothing\}.
\]
In particular, is it true that
\[
\tau(r)=1+\frac{r(r+1)}{2}
\qquad (r\ge 0)?
\]
Equivalently, does one have
\[
\Delta(n)=\max\Bigl\{r:1+\frac{r(r+1)}{2}\le n\Bigr\}?
\]
\end{problem}

\begin{problem}[First-occurrence families]\label{prob:first-occurrence-family}
Assume $r\ge 2$, and let
\[
\delta_r=(r,r-1,\dots,2,1)
\]
be the staircase partition of size $r(r+1)/2$. Is it true that the first-occurrence layer
\[
L_r(\tau(r))
\]
consists exactly of the partitions obtained from $\delta_r$ by adding one cell at an addable
corner?
\end{problem}

\begin{problem}[Boundary thresholds]\label{prob:boundary-threshold-final}
Determine the boundary threshold function
\[
\tau_{\partial}(r)=\min\{n:\partial_{r,r+1}(n)\neq\varnothing\}.
\]
In particular, does one have
\[
\tau_{\partial}(r)=\tau(r+1)
\qquad (r\ge 1)?
\]
That is, does every new simplex layer first appear already in direct contact with the previous one?
\end{problem}

\begin{problem}[Geometry of phase boundaries]\label{prob:boundary-geometry}
Describe the structure of the interface graph
\[
\mathcal I_{r,r+1}(n)
\]
and the geometry of the boundary
\[
\partial_{r,r+1}(n).
\]
For example:
\begin{enumerate}[label=\textup{(\roman*)},leftmargin=2em]
\item When is $\partial_{r,r+1}(n)$ connected?
\item How large can it be relative to $|L_r(n)|$ and $|L_{r+1}(n)|$?
\item Can the jump in local simplex dimension across an edge be arbitrarily large, or is it
uniformly bounded?
\end{enumerate}
\end{problem}

\begin{problem}[Placement relative to large-scale structures]\label{prob:placement}
Study the placement of the higher simplex layers and their phase boundaries relative to the
framework $\Fw_n$, the axis $\Ax_n$, the spine $\Sp_n$, and the central neighborhoods $C_n^{(t)}$.
More specifically, determine the asymptotic behaviour of the traces
\[
L_r(n)\cap \Fw_n,\qquad
L_r(n)\cap \Ax_n,\qquad
L_r(n)\cap \Sp_n,
\qquad
\partial_{r,r+1}(n)\cap C_n^{(t)},
\]
and of the corresponding concentration radii. Do higher simplex layers and their boundaries migrate
inward as $n$ grows?
\end{problem}

\begin{problem}[Extremal families]\label{prob:extremal-families}
Describe the partitions that maximize the local simplex dimension:
\[
\Delta(n)=\max_{\lambda\vdash n}\dimloc(\lambda).
\]
Are the extremal vertices always near staircase shapes, or do other families eventually dominate?
\end{problem}

\begin{problem}[Interaction with global topology]\label{prob:topology}
Clarify the relation between the increasing local simplex complexity of the layers and the
comparatively simple global homotopy type of the clique complex $K_n$. In particular, can one
explain more conceptually why the growth of local simplex dimension does not lead to a comparable
increase in global topological complexity?
\end{problem}

The simplex-layer viewpoint introduced here is intended as a structural tool rather than as a final
description. It isolates a new level of organization in the partition graph --- between purely
local transfer data and the already studied global geometric skeleton --- and suggests that the
emergence of higher-dimensional local simplex regimes is governed by a rigid threshold
geometry.

\section*{Acknowledgements}

The author acknowledges the use of ChatGPT (OpenAI) for discussion, structural planning, and
editorial assistance during the preparation of this manuscript. All mathematical statements,
proofs, computations, and final wording were checked and approved by the author, who takes full
responsibility for the contents of the paper.


\begin{thebibliography}{99}

\bibitem{partition-topology}
F.~B.~Lyudogovskiy,
\emph{The homotopy type of the clique complex of the partition graph},
arXiv:2603.14370 [math.CO], 2026.

\bibitem{partition-local}
F.~B.~Lyudogovskiy,
\emph{Local Morphology of the Partition Graph},
arXiv:2603.18696 [math.CO], 2026.

\bibitem{partition-growth}
F.~B.~Lyudogovskiy,
\emph{The Partition Graph as a Growing Discrete Geometric Object},
arXiv:2603.21221 [math.CO], 2026.

\bibitem{partition-axial}
F.~B.~Lyudogovskiy,
\emph{Axial Morphology of the Partition Graph: Self-Conjugate Axis, Spine, and Concentration},
preprint, 2026.

\bibitem{Savage}
C.~D.~Savage,
\emph{Gray code sequences of partitions},
J. Algorithms 10 (1989), no.~4, 577--595.

\bibitem{RSW}
D.~Rasmussen, C.~D.~Savage, and D.~B.~West,
\emph{Gray code enumeration of families of integer partitions},
J. Combin. Theory Ser. A 70 (1995), no.~2, 201--229.

\bibitem{Mutze}
T.~M\"utze,
\emph{Combinatorial Gray codes---an updated survey},
Electron. J. Combin. 30 (2023), no.~3, DS26.

\bibitem{Bal}
H.~S.~Bal,
\emph{Lognormal degree distribution in the partition graphs},
arXiv:2202.09819, 2022.

\bibitem{Adamaszek}
M.~Adamaszek,
\emph{Clique complexes and graph powers},
Israel J. Math. 196 (2013), no.~1, 295--319.

\bibitem{GoyalShuklaSingh}
S.~Goyal, S.~Shukla, and A.~Singh,
\emph{Topology of clique complexes of line graphs},
The Art of Discrete and Applied Mathematics 5 (2022), no.~2, \#P2.06.

\end{thebibliography}
\end{document}